\documentclass[a4paper,11pt,reqno]{amsart}
\usepackage[left=26mm,right=26mm,top=30mm,bottom=30mm]{geometry}

\usepackage{amssymb}
\usepackage[shortlabels]{enumitem}
\usepackage{bm}
\usepackage{mathtools}
\usepackage{mathrsfs}
\usepackage{array} 
\usepackage{xcolor} 
\usepackage{xparse}
\usepackage[sort&compress,numbers]{natbib}
\usepackage{caption}
\usepackage{subcaption}
\usepackage{booktabs}
\usepackage{circuitikz}

\usepackage{pstricks}
\usepackage{pst-node}
\usepackage{tikz}
\usetikzlibrary{arrows,chains,matrix,positioning,scopes}
\tikzstyle{element}=[rectangle,draw,fill=white, line width=1pt]
\tikzstyle{terminal}=[circle,draw, scale=0.3, line width=1pt,red]
\tikzstyle{fleche}=[->,>=stealth', very thick]
\tikzstyle{fleche1}=[->,>=stealth', very thick, red]
\usepackage{placeins}

\usepackage{bbm}
\def\1{\mathbbm{1}}

\def\la{\lambda}

\def\t{\tau}

\def\calB{{\calB}}

\def\calT{{\mathcal{T}}}

\def\calB{{\mathcal{B}}}

\def\calA{{\mathcal{A}}}

\def\calL{{\mathcal{L}}}
\def\R{\mathbf{R}}

\def\C{\mathbf{C}}

\def\N{\mathbf N}

\def\F{\mathbb F}

\newtheorem{theorem}{Theorem}[section]
\theoremstyle{plain}
\newtheorem{definition}{Definition}[section]
\newtheorem{lemma}{Lemma}[section]
\newtheorem{proposition}{Proposition}[section]

\newtheorem{remark}{Remark}[section]
\numberwithin{equation}{section}

\newtheorem{mainassumptions}{Assumption}[section]

\makeatletter
\newcommand\makebig[2]{%
	\@xp\newcommand\@xp*\csname#1\endcsname{\bBigg@{#2}}%
	\@xp\newcommand\@xp*\csname#1l\endcsname{\@xp\mathopen\csname#1\endcsname}%
	\@xp\newcommand\@xp*\csname#1r\endcsname{\@xp\mathclose\csname#1\endcsname}%
}
\makeatother

\makebig{biggg} {3.0}
\makebig{Biggg} {3.5}
\makebig{bigggg}{4.0}
\makebig{Bigggg}{4.5}
\makebig{biggggg}{5}
\makebig{Biggggg}{5.5}
\makebig{bigggggg}{6}
\makebig{Bigggggg}{6.5}

\makeatletter
\newcommand{\doublehat}[1]{%
	\begingroup%
	\let\macc@kerna\z@%
	\let\macc@kernb\z@%
	\let\macc@nucleus\@empty%
	\hat{\mathchoice%
		{\raisebox{.2ex}{\vphantom{\ensuremath{\displaystyle #1}}}}%
		{\raisebox{.2ex}{\vphantom{\ensuremath{\textstyle #1}}}}%
		{\raisebox{.16ex}{\vphantom{\ensuremath{\scriptstyle #1}}}}%
		{\raisebox{.14ex}{\vphantom{\ensuremath{\scriptscriptstyle #1}}}}%
		\smash{\hat{#1}}}%
	\endgroup%
}
\makeatother

\NewCommandCopy{\ordinaryexists}{\exists}
\RenewDocumentCommand{\exists}{}{\mathop{{}\ordinaryexists}}
\begin{document}
	\title{Well-posedness of boundary control systems and application to ISS for coupled heat equations with boundary disturbances and delays}
	\author{Yassine El Gantouh $^{\tt1,\ast}$, Jun Zheng $^{\tt1}$, and Guchuan Zhu $^{\tt2}$}
	
	\thanks{
		\vspace{-1em}\newline\noindent
		\newline\noindent 
		$^{\tt1}$ School of Mathematics, Southwest Jiaotong University, Chengdu, 611756, Sichuan, China.
     	$^{\tt2}$  Department of Electrical Engineering, Polytechnique Montréal, Montreal, H3T 1J4, Quebec,
     	Canada.
		\newline\noindent 
		{\sc Emails}:
		{\tt elgantouhyassine@gmail.com}, ~{\tt zhengjun2014@aliyun.com}, ~{\tt guchuan.zhu@polymtl.ca}.
		\newline\noindent
		$^{\ast}$ Corresponding author.
	}
	
	\begin{abstract}
    This paper studies the existence of solutions and, in particular, the well‑posedness of a class of boundary control systems. Our main result provides explicit and verifiable conditions on the system data that guarantee continuous dependence of solutions on the initial data and $L^p$-inputs. The proof relies on a new boundedness estimate for the input/output maps of linear time‑invariant infinite‑dimensional systems with unbounded control and observation operators. The developed technique is applied to derive specific conditions for the exponential input‑to‑state stability of boundary‑coupled heat equations with boundary disturbances and time‑delays.
	\end{abstract}
	\maketitle
	
	\pagestyle{myheadings} \thispagestyle{plain} \markboth{\sc  Y.\ EL Gantouh, J.\ Zheng, G.\ Zhu}{\sc}

	\section{Introduction}
     The well-posedness and long-term behavior of boundary control systems have been extensively investigated in the literature. The pioneering work of \cite{Fattorini1968} provided the first systematic treatment of this class of infinite-dimensional systems, establishing a foundation that has been built upon by numerous researchers. Subsequent significant contributions, such as those by \cite[Sect. 3.3]{Curtain1995}, \cite{Emirsajlow2000}, \cite[App. 3B]{Lasiecka2000}, \cite[Chap. 11]{Jacob2012}, and \cite[Chap. 10]{TW}, have further enriched the theory. A central element emerging from this body of work is the formulation of such systems as abstract evolution equations, which provides a powerful framework for analysis.
     \vspace{.1cm}
     
     This established framework typically models linear time-invariant (LTI) infinite-dimensional systems as an abstract boundary control problem of the form:
     \begin{align}\label{BCS}
     	\begin{cases}
     		\dot{z}(t)= \tilde{A}z(t),& t> 0,\qquad z(0)=z_0,\\
     		\Upsilon z(t)=Ku(t), & t> 0,
     	\end{cases}
     \end{align}
     where the state variable $ z(\cdot) $ evolves in a Banach space $ X $ with the norm $\|\cdot\|$, $z_0\in X$ is the initial state, and the input $u: \mathbb{R}_+ \to U$ is a locally $p$-integrable function for some $p \in [1,\infty]$. The system's internal dynamics are governed by the closed and densely defined operator $\tilde{A}: {D}(\tilde{A}) \subset X \to X$. The boundary conditions are expressed through a continuous linear operator $ \Upsilon: {D}(\tilde{A}) \to V$, and the control operator $K: U \to V$ maps inputs to actions on the boundary. The spaces $X$, $V$, and $U$ represent the state, boundary, and input spaces, respectively.
     \vspace{.1cm}
     
     While this model has proven its effectiveness, it is not sufficiently general to capture certain complex physical phenomena, such as those involving dynamic boundary feedback or complex coupling. In this paper, we extend the framework of the abstract boundary control system \eqref{BCS} by replacing the single boundary operator $\Upsilon$ with a difference of two operators $G,\Gamma, {D}(\tilde{A}) \to V$, i.e., $\Upsilon = G-\Gamma$, where $\Gamma$ is not necessarily closed or even closable. This leads to a more general initial/boundary-value problem:
     \begin{align}\label{Main-system}
     	\begin{cases}
     		\dot{z}(t) = \tilde{A} z(t), & t > 0, \quad z(0)=z_0, \\
     		G z(t) = \Gamma z(t) + K u(t), & t > 0.
     	\end{cases}
     \end{align}
     \vspace{.1cm}
     
     From a PDE point of view, continuous dependence on data--the essence of well‑posedness--is a fundamental property. It is essential not only for establishing existence, uniqueness, and regularity of solutions to evolution equations, but also for analyzing stability and robustness concepts such as input‑to‑state stability (ISS) \cite{JaNPS,JaSZ,Karafyllis2019a,Lhachemi2019,Andri,SZZ,CZZ}. Therefore, our main objective is to establish direct and explicit conditions on the operators $\tilde{A},G,\Gamma$, and $K$ such that the  system  \eqref{Main-system} is well-posed in the following sense: for any initial datum $z_0\in X$ and any input $u\in {L}^p_{loc}(\R_+;U)$, system \eqref{Main-system} admits a unique solution $z \in C(\R_+;X)$ and for every $\tau>0$, there exists a constant $c_{\tau,p}>0$ such that
     \begin{align*}
     	\Vert z(\t)\Vert \le c_{\tau,p}\left(\Vert z_0\Vert +\Vert u\Vert_{L^p(0,\t;U)}\right).
     \end{align*}
     \vspace{.1cm}
     
     When the operator $\Gamma$ is bounded, the well‑posedness of \eqref{Main-system} follows directly from classical results on the boundary control system \eqref{BCS}. However, the case where the operator $\Gamma$ is unbounded and non‑closable is significantly more challenging. Prior works (see, e.g., \cite{El2, HMR})  have addressed this challenge by employing the feedback theory developed in \cite{Salam,Sta,WF}. This approach reformulates \eqref{Main-system} as an input/output system interconnected with a static feedback operator.  Within this framework, well‑posedness (in the sense of \cite{WF}) is contingent upon the following three abstract conditions: \begin{itemize}
     	\item[(a)] an abstract operator triple derived from $\tilde{A}$, $G$, and $\Gamma$ is well-posed; 
     	\item[(b)] the transfer function of this abstract triple is proper and has a strong limit at $+\infty$ (along the real axis); and  
     	\item[(c)] the static feedback operator is admissible.
     \end{itemize}
     A key limitation of this method is its dependency on conditions that are not expressed in terms of the original system data. Moreover, verifying these abstract conditions is not a trivial task, particularly in a Banach space setting. For instance, establishing condition (a) requires proving the boundedness of the associated state/output, input/state, and input/output maps \cite{Salam,Sta,Tucsnak2014,WF}. This, in turn, presents a serious obstacle, as boundedness often depends on intrinsic properties of the underlying semigroup that are rarely available in explicit form.
     \vspace{.1cm}
     
     Motivated by these limitations, this paper investigates the well-posedness of the initial/boundary value control system \eqref{Main-system} without directly invoking abstract admissibility conditions. Our main result, Theorem \ref{Main1}, provides explicit and verifiable criteria for the well-posedness of system \eqref{Main-system}. A crucial ingredient in its proof is a novel boundedness estimate for input/output maps, presented in Theorem \ref{S5.P1}. The derivation of this estimate exploits the positivity of the unforced dynamics (i.e., $u \equiv 0$) combined with the boundedness and positivity of solutions to a related elliptic problem. Notably, Theorem \ref{S5.P1} extends the method of \cite{ChM} to the Banach space setting. Furthermore, we show that our assumptions are sufficient to ensure that the abstract conditions (b) and (c) are also satisfied, thereby establishing the well-posedness of system \eqref{Main-system}.
     
     The rest of the paper is organized as follows. Section \ref{Sec:2} presents our main results, {namely,  Theorem  \ref{Main1} and Theorem \ref{S5.P1}}. In Section \ref{Sec:3}, we review relevant results on positive LTI systems. Section \ref{Sec:4} {presents} the proofs of {the} main theorems. In Section \ref{Sec:5}, we apply our results to derive specific conditions for the exponential ISS of a system of coupled heat equations with time delays and boundary disturbances. Some concluding remarks are provided in Section \ref{Sec:6}. Finally, Appendix \ref{App} provides two technical lemmas used in the ISS analysis of the example treated in Section \ref{Sec:5}.
     \vspace{.1cm}
     
     \textbf{Notation.} Throughout this paper, $\mathbb{C}$, $\mathbb{R}$, $\mathbb{R}_+$, and $\mathbb{N}$ denote the sets of complex numbers, real numbers, positive real numbers, and natural numbers, respectively. Let $E$ be a Banach space, $p\in [1,\infty]$, and $J \subset \mathbb{R}$ an interval. We denote by $L^p(J;E)$ the space of Bochner measurable, $p$-integrable functions $f \colon J \to E$;  $C(J;E)$ the space of continuous $E$-valued functions on $J$; and  $W^{m,p}(J;E)$ the $E$-valued Sobolev space of order $m\in \N$, consisting of all $f\in L^p(J;E)$ whose weak derivatives $\partial_x^{(k)}f$ belong to $L^p(J;E)$ for all $k=1,2,\ldots,m$.
     \vspace{.1cm}
     
     Let $E$ and $F$ be Banach spaces. The space of bounded linear operators from $E$ to $F$ is denoted by $\mathcal{L}(E,F)$, and we write $\mathcal{L}(E)=\mathcal{L}(E,E)$. For $P\in\mathcal{L}(E,F)$, its restriction to a subspace $Z\subset E$ is written $P|_Z$; $\operatorname{ran}P$ and $\ker P$ stand for its range and kernel, respectively. The identity operator is denoted by $I$; the underlying space will be clear from the context.
     \vspace{.1cm}
     
     A \emph{Banach lattice} $(E,\le)$ is a partially ordered Banach space such that every pair $x,y\in E$ admits a supremum $x\vee y$ and, for all $x,y,z\in E$ and $\alpha\ge0$,
     \begin{itemize}
     	\item 	$x \leq y $ implies $(x + z \leq y + z \quad \text{and} \quad \alpha x \leq \alpha y)$.
     	\item $|x| \leq |y|$ implies $\|x\| \leq \|y\|$.
     \end{itemize}
     Here, the absolute value is defined by $|x| = x \vee (-x)$. An element $x\in E$ is \emph{positive} if $x\ge0$; the set of positive elements forms the \emph{positive cone} $E_+$. This notation is also used for general ordered Banach spaces. The topological dual $E'$ equipped with the dual norm and order is again a Banach lattice. An operator $P\in\mathcal{L}(E,F)$ is called \emph{positive} if $P E_+\subset F_+$, and the collection of such operators is denoted $\mathcal{L}_+(E,F)$. For further details we refer to \cite{CHARALAMBOS} or \cite{Schaf}.

	\section{Main results}\label{Sec:2}
     This section presents the main results of the paper. We begin by recalling some basic definitions. Let $A$ be the generator of a $C_0$-semigroup $T:= (T(t))_{t \geq 0}$ on a Banach space $X$. The uniform growth bound of $T$ is defined as $\omega_0(T):=\inf\{t>0: t^{-1}\log(\Vert T(t)\Vert)\}$. The resolvent set of $A$ is defined as
      \begin{eqnarray*}
     \rho(A) := \{\lambda\in \mathbb{C}:\lambda I-A \text{ is bijective and }
     		(\lambda I - A)^{-1} \in \mathcal{L}(X)\},
     \end{eqnarray*}
     and the spectrum is $\sigma(A) := \mathbb{C} \setminus \rho(A)$. For $\lambda \in \rho(A)$, the operator $R(\lambda,A) := (\lambda I-A)^{-1}$ is called the resolvent operator of $A$ at $\lambda$. The spectral bound of $A$ is defined as
     \begin{align*}
     	s(A) := \sup\{\operatorname{Re} \lambda : \lambda \in \sigma(A)\},
     \end{align*}
     with the convention that $s(A)= -\infty$ if $\sigma(A) = \emptyset$. For the generator $A$ of a $C_0$-semigroup $T$, it holds that $-\infty\le s(A) \le \omega_0(T)<+\infty $ (see, e.g., \cite{engel2000}). The semigroup $T$ is called uniformly exponentially stable if $\omega_0(T) < 0$. 
     \vspace{.1cm}
     
     We now introduce the following standard assumptions.   
     \begin{mainassumptions}\label{Assp1}
     	\begin{itemize}
     		\item[(i)] $A:=\tilde{A}\vert_{\ker G}$ generates a $C_0$-semigroup $T$ on $X$; and 
     		\item[(ii)] $G$ is surjective.
     	\end{itemize}
     \end{mainassumptions}
     Under Assumption \ref{Assp1}, the domain of $\tilde{A}$ admits the decomposition
     \begin{eqnarray}\label{decomposition}
     	{D}(\tilde{A})={D}(A) \oplus \ker (\la I-\tilde{A}),\qquad \la\in \rho(A).
     \end{eqnarray}
     Moreover, the Dirichlet operator associated with $(\tilde{A},G)$,
     \begin{eqnarray*}
     	D_\lambda:=\left(G\vert_{\ker(\lambda I-\tilde{A})}\right)^{-1}: V\to X, \quad \lambda\in\rho(A),
     \end{eqnarray*} 
     exists and is bounded; see \cite[Lems. 1.2 and 1.3]{Gr}. We note that $D_\lambda$ is the solution operator of the abstract elliptic problem
     \begin{align*}
     	\tilde{A}x=\lambda x,\qquad Gx=v,\qquad \la\in \C,
     \end{align*}
     see \cite[Prop. 2.8, eqn. 2.20]{Salam}. For $\lambda\in \rho(A)$, define
     \begin{eqnarray}\label{boundary-control}
     	B:=(\lambda I-A_{-1})D_\lambda\in\mathcal{L}(V,X_{-1}^A).
     \end{eqnarray}
     One can verify that $B$ is independent of $\lambda$ and satisfies $\operatorname{ran} B \cap X = \{0\}$ and
     \begin{eqnarray}\label{representation}
     	(\tilde{A}-A_{-1})|_{{D}(\tilde{A})}=B G.
     \end{eqnarray}
     Here and in the following, $X_{-1}^A$ denotes the extrapolation space associated with $A$, and $A_{-1}$ (with domain ${D}(A_{-1})=X$) is the generator of the extrapolated semigroup $T_{-1} := (T_{-1}(t))_{t \ge 0}$, cf. \cite[Chap. II.5]{engel2000}.
     \vspace{.1cm}
     
     Throughout the paper, we assume that $X$ and $V$ are Banach lattices and that $T$ is a positive semigroup, i.e., $T(t)X_+ \subset X_+$ for all $t \ge 0$. To extend the notion of positivity to the extrapolation space $X_{-1}^A$, we follow \cite[Def. 2.1]{BJVW} and define the positive cone $X_{-1,+}^A$ as the closure of $X_+$ in the norm $\|\cdot\|_{-1}^A$. Then, we have $X_+ \subset X_{-1,+}^A$. Moreover, if $X$ is a real Banach lattice, then by \cite[Prop. 2.3]{BJVW} we have $X_+ = X \cap X_{-1,+}^A$. For further details, see \cite[Sect. 2.2]{SGPS}.
     \vspace{.1cm}
     
     In addition to Assumption \ref{Assp1}, we also introduce the following assumptions.
     \begin{mainassumptions}\label{Assp2}
     	\begin{itemize}
     		\item[(i)] $D_\la$ is positive for all sufficiently large $\la\in \R$.
     		\item[(ii)] $\Gamma$ is positive, and for some $\alpha>0$ there exists $\gamma_\alpha>0$ such that
     		\begin{eqnarray*}
     			\int_{0}^{\alpha}\Vert \Gamma T(t)x\Vert^p_V \mathrm{d}t\leq \gamma_\alpha^p\Vert x\Vert^p_X, \quad \forall x\in {D}_+(A),
     		\end{eqnarray*}
     		where ${D}_+(A):={D}(A)\cap X_+$ and $p\in [1,\infty)$.
     		\item[(iii)] There exists $\la_0>\omega_0(T)$ such that
     		\begin{eqnarray*}
     			\sup_{\la\ge \la_0}\Vert \la D_\la v\Vert_X<\infty, \quad \forall v\in V.
     		\end{eqnarray*}
     		\item[(iv)] For all $v\in V$ and $u^*\in V'$:
     		\begin{eqnarray*}
     			\lim_{\R\ni \la\to +\infty }\langle \Gamma D_\la v,u^*\rangle_{V,V'}=0,
     		\end{eqnarray*}
     		where $\langle\cdot,\cdot\rangle_{V,V'}$ {stands} for the duality pairing between $V$ and $V'$.
     	\end{itemize}
     \end{mainassumptions}
     \begin{remark}\label{Further-details}
     	It should be noted that Assumption \ref{Assp2}-(iii) holds only if the state-space $X$ is nonreflexive. Indeed, if $X$ were reflexive, the unit ball would be weakly sequentially compact. Consequently, a subsequence $(n_k D_{n_k} v)_k$ would converge weakly in $X$. However, the whole sequence $(n D_n v)_n$ converges to $B v$ in $X_{-1}^A$. This implies $B v \in X$, and by the closed graph theorem, $B \in \mathcal{L}(V, X)$, contradicting the fact that $\operatorname{ran} B \cap X = \{0\}$.
     \end{remark}
     
     The following is our {first} main result, which {provides} direct and explicit conditions on the operators in \eqref{Main-system} that guarantee its well-posedness.
     \begin{theorem}\label{Main1}
     	Let $U$ be a Banach space, $K \in \mathcal{L}(U,V)$, and $p\in [1,\infty)$. Let Assumptions \ref{Assp1}-\ref{Assp2} be satisfied and $q\in [p,\infty]$. Then, the operator $\calA:{D}(\calA)\subset X \to X$ defined by 
     	\begin{eqnarray}\label{def:A}
     		\calA:= \tilde{A}, \qquad {D}(\calA):= \left\{x\in {D}(\tilde{A}):\ G x=\Gamma x\right\},
     	\end{eqnarray}
     	generates a positive $C_0$-semigroup $\calT:=(\calT(t))_{t\ge 0}$ on $X$. Moreover, for every initial condition {$z_0\in X$} and every input function $u\in {L}^q_{loc}(\R_+,U)$, the initial/boundary-value control problem \eqref{Main-system} has a unique mild solution $z(\cdot)\in { C}(\R_+;X)$ given by 
     	\begin{eqnarray}\label{variation1}
     		z(t)=\calT(t){z_0}+\int_{0}^{t} \calT_{-1}(t-s)B K u(s)\ {\mathrm{d}s}, \quad t\ge 0.
     	\end{eqnarray}
     \end{theorem}
     
     The proof of Theorem \ref{Main1} relies on the following theorem, which is the second result obtained in this paper.
     
     \begin{theorem}\label{S5.P1}
     	Under the assumptions of Theorem \ref{Main1}, we have 
     	\begin{eqnarray}\label{Estimate-wellposed}
     		\left\Vert \Gamma\int_{0}^{t}T_{-1}(t-s)B v(s) \ {\mathrm{d}s}\right\Vert_{{L}^{p}(0,\t;V)}\leq \eta_\tau \Vert v\Vert_{{L}^{p}(0,\t;V)},
     	\end{eqnarray}
     	for all $  v\in {W}^{1,p}_{0,+}(0,\t;V):=\{ v\in {W}^{1,p}_+(0,\t;V): v(0)=0\}$ and $ \t >0 $, where $ \eta_\tau>0 $ is independent of $v$ and satisfies $ \lim_{\tau\to 0}\eta_\tau=0 $ for all $p>1$. 
     \end{theorem}
     
     \begin{remark}
     	The result above extends the method of \cite[Thm. 3.2]{ChM} to Banach lattices. Notably, Theorem \ref{S5.P1} provides an alternative way to establish the boundedness of input/output maps for LTI systems with unbounded control and observation operators. 
     	\vspace{.1cm}
     	
     	It is worth noting that a recent result by \cite{Alessio2026}, exploiting Banach lattice theory and the positivity of the unforced  dynamic, derives a spectral condition that also implies the boundedness of input/output maps when the input and output spaces are finite-dimensional.
     \end{remark}

    \section{Preliminaries on linear infinite-dimensional positive systems}\label{Sec:3}
     In this section, we provide a brief preliminaries on LTI positive systems with unbounded input and output operators. For this purpose, let $X$ and $V$ be Banach lattices. Consider the following input/output system
     \begin{eqnarray}\label{input-ouptut}
     	\begin{cases}
     		\dot{z}(t) =A_{-1} z(t)+Bv(t),& t> 0,\quad {z(0)=z_0},\\
     		y(t) = P z(t), & t> 0,
     	\end{cases}
     \end{eqnarray}
     where $A$ generates a $C_0$-semigroup $T$ on $X$, $B \in \calL(V, X_{-1}^A)$ is the input operator, and $P \in \calL(Z,V)$ is the output operator. Here, the space $Z$ is defined by
     \begin{align*}
     	Z=X_1+R(\la,A_{-1})BV, \qquad \la\in \rho(A),
     \end{align*}
     where $X_1$ denotes ${D}(A)$ endowed with the graph norm. Equipping $Z$ with the norm
      \begin{align*}
     		\Vert z\Vert^2_{Z}=\inf\left\{\Vert x\Vert^2_{X_1}+\Vert v\Vert^2_{V}:x\in X_1,v\in V,z=x+R(\la,A_{-1}) Bv \right\},
     \end{align*}
     makes it a Banach space satisfying $X_1 \subseteq Z \subseteq X$ with continuous embeddings. We note that $Z$ is independent of the choice of $\lambda$ due to the resolvent equation.
     \vspace{.1cm}
     
     We recall the following definition (see, e.g., \cite[Def. 3.3]{El3}).
     \begin{definition}\label{Definition-Lp-well-posed}
     		The input/output system \eqref{input-ouptut} is called ${L}^p$-well-posed (for $1\le p<\infty$) if, for every $ \tau > 0 $, initial state ${z_0}\in {D}(A)$, and input $v\in {W}^{1,p}_{0}(0,\tau;V)$, there exists a constant $ c_\tau>0 $ (independent of $x$ and $v$) such that the following estimate holds for all solutions $(z,y)$ of \eqref{input-ouptut}:
     		\begin{eqnarray*}
     			\Vert z(\tau)\Vert^{p}_{X}+\Vert y\Vert^{p}_{ {L}^{p}(0,\tau;V)}\leq c_\tau\left( \Vert {z_0}\Vert^{p}_{X}+\Vert v\Vert^{p}_{ {L}^{p}(0,\tau;V)}\right).
     		\end{eqnarray*}
     	\end{definition}
     	
     	The following lemma provides a complete characterization for the well-posedness of \eqref{input-ouptut} when the involved operators are positive.
     	\begin{lemma}[{\cite[Prop. 3.3]{El3}}]\label{Well-posed-system}
     		Let $T$, $B$, and $P$ be positive operators. Set $C := P\vert_{{D}(A)}$ and let $p \in [1,\infty)$. Then, the input/output system \eqref{input-ouptut} is ${L}^p$-well-posed if and only if
     		\begin{itemize}
     			\item[(i)] $ B$ is an ${L}^p$-admissible positive control operator for $A$, i.e., {for} some (hence for all) $t>0$ and all $v\in {L}_+^p(0,t;V)$, we have
     			\begin{eqnarray*}
     				\Phi_{t}v:=\int_{0}^{t} T_{-1}(t-s)Bv(s)\ {\mathrm{d}s}\in X_+;
     			\end{eqnarray*}
     			\item[(ii)] $C$ is an ${L}^p$-admissible positive observation operator for $ A $, i.e.,  {for} some (hence for all) $t >0,$ there exists $\gamma_t>0$ such that
     			\begin{eqnarray}\label{observatio-estimate}
     				\Vert CT(\cdot)x\Vert_{{L}^{p}(0,t;V)}\leq \gamma_t\Vert x\Vert_X,\quad \forall x\in {D}_+(A);
     			\end{eqnarray}
     			\item[(iii)] For every $t >0$ there exists $\eta_t>0 $ such that
     			\begin{eqnarray*}
     				\Vert P \Phi_\cdot v\Vert_{{L}^{p}(0,t;V)}\leq \eta_t \Vert v\Vert_{{L}^{p}(0,t;V)},\; \forall v\in {W}^{1,p}_{0,+}(0,t;V).
     			\end{eqnarray*}
     		\end{itemize}
     	\end{lemma}
     	
     	We comment on some consequences of the above lemma. Condition (i) implies that the state trajectories $z(\cdot)$ satisfies 
     	\begin{eqnarray*}
     		0\le	z(t)= T(t){z_0}+ \Phi_tv ,\;\; \forall t \ge 0, {z_0} \in X_+, v \in {L}^{p}_+(\R_+; V).
     	\end{eqnarray*}
     	Condition (ii) implies that \eqref{observatio-estimate} holds for all $x \in X$ \cite[Rem. 3.2]{El3}. Thus, the map $x \mapsto C T(\cdot) x$, defined for $x \in {D}(A)$, has a positive bounded extension from $X$ to ${L}^p(0,t;V)$ for each $t > 0$.	By assuming $0\in \rho(A)$ (without loss of generality) and using integration by parts along with the admissibility of $B$, we obtain
     	\begin{eqnarray}\label{phi}
     		\Phi_t v =(-A)^{-1}Bv(t)-(-A)^{-1}\Phi_t\dot{v}\in Z,
     	\end{eqnarray}
     	for all $v\in {W}^{1,p}_0(0,t;V)$. Therefore, we define the operator
     	\begin{eqnarray}\label{input-output-operator}
     		(\mathcal{F}v)(s):=P\Phi_s v, \quad \forall s\in [0,t],\ v\in  {W}^{1,p}_0(0,t;V).
     	\end{eqnarray}
     	By condition (iii) and the density of ${W}^{1,p}_{0,+}(0,t;V)$ in ${L}^{p}_+(0,t;V)$, the operator $\mathcal{F}$ extends uniquely to a positive bounded linear operator from ${L}^{p}(0,t; V)$ to ${L}^{p}(0,t; V)$ for each $t > 0$.

     	\begin{definition}\label{Def.triple}
     		Consider the setting of Lemma \ref{Well-posed-system}. We say that $(A,B,C)$ is a positive ${L}^p$-well-posed triple on ($V,X,V$) if the conditions (i)-(iii) in Lemma \ref{Well-posed-system} are satisfied. 
     	\end{definition}
     	
     	We note that for a positive ${L}^p$-well-posed triple $(A,B,C)$, there exists a unique analytic $\mathcal{L}(V)$-valued transfer function $H$ given by
     	\begin{align*}
     		H(\lambda) = P R(\lambda,A_{-1}) B, \qquad \forall  {\mathrm{R}\mathrm{e}( \lambda)} > \omega_0(T).
     	\end{align*}
     	It is important to note that $P$ denotes an extension of the observation operator 
     	$C$, which is generally not unique. Consequently, the transfer function is defined up to such an extension; while we do not notationally distinguish between different versions that differ by a fixed operator, as all such functions are analytic. For more details, see \cite[Sect. 3]{StaffansWeiss2002} or \cite[Sect. 4.6]{Sta}.
     	\vspace{.1cm}
     	
     	Now, we recall subclass of ${L}^p$-well-posed linear systems; see \cite[Def. 5.6.1]{Sta} (see also \cite{WR}).
     	\begin{definition}\label{Definition-regularity}
     		Consider the setting of Definition \ref{Def.triple}. The triple $(A,B,C)$ is called a positive $L^p$-well-posed triple with zero feedthrough which is of weakly, strongly, or uniformly regular type if, respectively, we have for any $v\in V$, $u^*\in V'$, $\lim_{\la\to +\infty }\langle H(\la)v,u^*\rangle_{V,V'}=0$, for any $v\in V$, $\lim_{\la\to +\infty }\Vert H(\la)v\Vert_V=0$, or $\lim_{\la\to +\infty }\Vert H(\la)\Vert_{\calL({V})}=0$, all referring to purely real-valued $\la$.
     	\end{definition}
     	\begin{remark}\label{Yosida-extension}
     		It is worth noting that for a positive ${L}^p$-well-posed strongly regular triple $(A,B,C)$ with zero feedthrough, the transfer function admits the representation
     		\begin{align*}
     			H(\la)=C_\Lambda R(\la,A_{-1})B,\qquad \forall {\mathrm{R}\mathrm{e}(\lambda)} >\omega_0(T),
     		\end{align*}
     		where $C_\Lambda $ is the Yosida extension of $C$ with respect to $A$ defined by
     		\begin{align*}
     			&D(C_\Lambda):= \{x\in X: \lim_{\la \to +\infty}C\la R(\la ,A)x \text{ exists  in } V\},\cr
     			&C_\Lambda x := \lim_{\la \to +\infty}C\la R(\la ,A)x,\qquad x\in D(C_\Lambda).
     		\end{align*}
     		Furthermore, the extension of the operator $\F$, defined in \eqref{input-output-operator}, is given by 
     		\begin{align*}
     			(\mathcal{F} v)(t)=C_\Lambda \Phi_t v,  
     		\end{align*}
     		for all $v\in {L}^p_{loc}(0,\tau;V)$ and a.e. $t\ge 0$. In this case, $\mathcal{F}$ is called the extended input/output map of  $(A,B,C)$.
     	\end{remark}
     	
     	We close this section by a version of the  Weiss-Staffans perturbation theorem for positive systems.
     	\begin{proposition}[{\cite[Thm. 3.1]{El3}}]\label{theorem}
     		Consider the setting of Lemma \ref{Well-posed-system}. Assume that $(A,B,C)$ is a positive ${L}^p$-well-posed weakly regular triple with zero feedthrough and the extended input/output map $\mathcal{F}$. Define 
     		\begin{align*}
     			W_{-1}&:=(\la I-\check{A}_{-1}){D}(C_\Lambda), \qquad \la\in \rho(A),
     		\end{align*}
     		where $\check{A}_{-1}$ is the restriction of $A_{-1}$ to ${D}(C_\Lambda)$. If $r(\mathcal{F})<1$, then the operator 
     		\begin{align*}
     			A^{cl}:=\check{A}_{-1}+BC_\Lambda,\qquad  {D}(A^{cl}):=\left\{x\in {D}(C_\Lambda):\ (\check{A}_{-1}+BC_\Lambda)\in X\right\},
     		\end{align*}
     		generates a positive $C_0$-semigroup $T^{cl}:=(T^{cl}(t))_{t\ge 0}$ on $X$ given by 
     		\begin{eqnarray}\label{variation}
     			T^{cl}(t)x=T(t)x+\int_{0}^{t}T_{-1}(t-s)BC_{\Lambda}T^{cl}(s)x\ {\mathrm{d}s},
     		\end{eqnarray}
     		for all $t\ge 0$ and $x\in X$.	Moreover, $(A^{cl},B^{cl},C^{cl}_\Lambda)$ is a positive ${L}^p$-well-posed strongly regular triple with
     		\begin{eqnarray*}
     			C^{cl}_\Lambda=C_\Lambda\in \calL({D}(C_\Lambda),V), \qquad B^{cl}=B\in \calL(V,W_{-1}).
     		\end{eqnarray*}
     		In addition, we have $s(A^{cl})\ge s(A)$.
     	\end{proposition}

	    \section{Proof of the main results}\label{Sec:4}
	     This section provides the proofs of the results stated in {Section \ref{Sec:2}}. 
	     
	     \subsection{Proof of Theorem \ref{S5.P1}}
	     We begin by recalling a result on admissibility of control operators. The proof follows from \cite[Thm. 9]{Desch} and \cite[Prop. 3.3]{NaS}, {and hence is omitted here.}

	     \begin{lemma}\label{lemApp}
	     	Let $X,U$ be Banach spaces and let Assumption \ref{Assp1} be satisfied. Suppose that there exists $\lambda_0 > \omega_0(T)$ such that $\sup_{\lambda \ge \lambda_0} \| \lambda D_\lambda v \|_X < \infty$ for all $v \in V$. Then, the operator $B$, defined by \eqref{boundary-control}, is an ${L}^1$-admissible control operator for $A$.
	     \end{lemma}
	     \begin{remark}
	     	We note that the above admissibility result is particularly interesting when the state-space $X$ is not reflexive, cf. Remark \ref{Further-details}.
	     \end{remark}
	     Let us also recall that
	     \begin{align}
	     	&\lim_{n\to +\infty} \Vert n R(n,A)x-x\Vert_{X_1}=0,\quad\forall x\in {D}(A),\label{Eq1}\\
	     	& \lim_{n\to +\infty} \Vert n R(n,A)x-x\Vert_{X}=0,\quad\forall x\in X,\label{Eq2}
	     \end{align}
	     see, e.g., \cite[Lem. II.3.4]{engel2000}. 
	     \vspace{.1cm}
	     
	     We now prove Theorem \ref{S5.P1} in two steps.
	     \vspace{.1cm}
	     
	     \underline{\bf Step 1.} We prove that 
	     \begin{eqnarray}\label{S5.9"}
	     	\lim_{n\to +\infty} \Vert \Gamma_n z-\Gamma z\Vert_{V}=0,\qquad \forall z\in {D}(\tilde{A}),
	     \end{eqnarray}
	     where, for sufficiently large $n\in \N$ ($n>s(A)$), $ \Gamma_n\in \mathcal{L}(X,V)$ {denotes} the Yosida-like approximations of $\Gamma$, defined by 
	     \begin{eqnarray*}
	     	\Gamma_n x:=\Gamma n R(n,A)x,\quad x\in X.
	     \end{eqnarray*}
	     Indeed, let $v\in V_+$. Using the resolvent equation, we obtain
	     \begin{align*}
	     	0\le \Gamma D_{\la} v\le \Gamma D_\mu v, \qquad \forall \la\ge \mu>s(A),
	     \end{align*}
	     {due to the facts that} $\Gamma$, $R(\la,A)$, and $D_\la$ are positive operators. Thus, the sequence $(\Gamma D_{n}v)$ is decreasing monotonically. Assumption \ref{Assp2}-(iv) further implies that $(\Gamma D_{n}v)$ decreases monotonically and weakly converges to $0$. Therefore, by Dini’s theorem for Banach lattices (see, e.g., \cite[Prop. 10.9]{BFR}), we have 
	     \begin{eqnarray}\label{regularity}
	     	\lim_{n\to +\infty } \Vert \Gamma D_{n}v\Vert_V=0,
	     \end{eqnarray}
	     for all $v\in V_+$ and hence for all $v\in V$. Thus, using the identity
	     \begin{eqnarray*}
	     	n\Gamma R(n,A) D_\la =\frac{1}{1-\tfrac{\la}{n}}\left(\Gamma D_\la-\Gamma D_n\right),\quad (\rho(A)\ni\la\neq n)
	     \end{eqnarray*}
	     together with \eqref{regularity}, we obtain
	     \begin{eqnarray}\label{nGamma}
	     	\lim_{n \to +\infty} \Vert n \Gamma R(n,A) D_\la v-\Gamma D_\la v \Vert_{V}= 0, \quad \forall v\in V.
	     \end{eqnarray}
	     
	     Now, from the decomposition \eqref{decomposition}, for any $z \in {D}(\tilde{A})$ there exist $x\in {D}(A)$ and $v \in V$ such that $z=x+D_\lambda v$ for some fixed $\lambda\in \rho(A)$. Therefore,
	     \begin{align*}
	     	\Vert \Gamma_n z-\Gamma z\Vert_{V}&\le \Vert \Gamma_n x-\Gamma x\Vert_{V}+\Vert \Gamma_n D_{\la} v-\Gamma D_\la v\Vert_{V}\\
	     	&\le \Vert \Gamma\Vert_{\calL({D}(\tilde{A}),V)} \Vert n R(n,A)x- x\Vert_{{D}(\tilde{A})}+ \Vert n \Gamma R(n,A) D_\la v-\Gamma D_\la v\Vert_{V}\\
	     	&\le \Vert \Gamma\Vert_{\calL({D}(\tilde{A}),V)} \Vert n R(n,A)x- x\Vert_{X_1}+\Vert n \Gamma R(n,A) D_\la v-\Gamma D_\la v\Vert_{V}.
	     \end{align*}
	     Taking the limit as $n \to +\infty$ and applying \eqref{Eq1} and \eqref{nGamma}, we obtain \eqref{S5.9"}.
	     \vspace{.1cm}
	     
	     \underline{\bf Step 2.} We prove the estimate \eqref{Estimate-wellposed}.\\ 
	     
	     For $t \ge 0$, let $\Phi_t^A\in \calL({L}^p(\mathbb{R}_+;V),X_{-1}^A)$ be the input-map of $(A,B)$, defined for every $v \in {L}^p(\mathbb{R}_+;V)$ by
	     \begin{eqnarray*}
	     	\Phi_t^A v:=\int_{0}^{t} T_{-1}(t-s)B v(s) {\mathrm{d}s}.
	     \end{eqnarray*}
	     Since $T$ is positive and $D_\la$ is positive for all $\la>s(A)$, it follows from \cite[Prop. 4.3]{SGPS} and \cite[Lem. 2.1]{elgantouh2024admissibility} that $\Phi_t^A v \in X_{-1,+}^A$ for all $t \ge 0$ and $v \in {L}^p_+(\mathbb{R}_+;V)$. Using Assumption \ref{Assp2}-(iii) and Lemma \ref{lemApp}, we obtain $\Phi_t^A v\in X_+$ for all $v \in {L}^1_+(\mathbb{R}_+; V)$ and $t\ge 0$. Therefore, by Hölder’s inequality, $B$ is an $L^p$-admissible positive control operator for $A$.
	     
	     Now, let $\tau > 0$, $v \in {W}^{1,p}_{0,+}(0,\tau; V)$, and let $n, m \in \mathbb{N}$ be sufficiently large. Using Assumption \ref{Assp2}-(ii), Hölder's inequality, and Fubini's theorem, we obtain
	      \begin{align*}
	     		\left\Vert \Gamma_n mR(m,A)\Phi_\cdot^A v \right\Vert_{{L}^{p}(0,\t;V)}^{p}
	     		=&\int_0^\t \left\Vert C nR(n,A)mR(m,A)\int_0^t T_{-1}(t-s)B v(s)\ {\mathrm{d}s}\right\Vert^{p}\ \mathrm{d}t \\
	     		=&\int_0^\t \left\Vert \int_0^t CT(t-s)mR(m,A)nD_nv(s)\ {\mathrm{d}s}\right\Vert^{p} \mathrm{d}t \\
	     		\le &\int_0^\t t^{p-1}\int_0^t\left\Vert CT(t-s)mR(m,A)nD_nv(s)\right\Vert^p{\mathrm{d}s}\ \mathrm{d}t\\
	     		\le &\t^{p-1}\int_0^\t \int_0^\t \left\Vert CT(t)mR(m,A)nD_nv(s)\right\Vert^p\mathrm{d}t\ {\mathrm{d}s}\\
	     		\le &\gamma_\t^p\t^{p-1}\int_0^\t \left\Vert mR(m,A)nD_nv(s)\right\Vert^p\ {\mathrm{d}s},
	     \end{align*}
     where $C:=\Gamma\vert_{{D}(A)}$. Passing to the limit as $m \to \infty$, using Fatou's lemma, and taking into account \eqref{Eq2} and Assumption~\ref{Assp2}-(iii), we obtain
	     \begin{align*}
	     	\left\Vert \Gamma_n \Phi_\cdot^A v \right\Vert_{L^{p}(0,\t;U)}^{p}	\le \kappa_0^p \gamma(\t)^p\t^{p-1}\Vert v\Vert_{L^{p}(0,\t;V)}^p,
	     \end{align*}
	     where $\kappa_0=\sup_{n\ge N} n\|D_n\|$. By a similar argument as in \eqref{phi}, we have 
	     \begin{align*}
	     	\Phi_t^A v=(-A)^{-1}B v(t)-(-A)^{-1}\Phi_t^A\dot{v}\in {D}(\tilde{A}),
	     \end{align*}
	     for all $v\in {W}^{1,p}_{0}(0,\t;V)$. Using Fatou's lemma and \eqref{S5.9"}, we obtain
	     \begin{eqnarray}\label{Estimate-C}
	     	\left\| \Gamma \Phi_\cdot^A v \right\|_{{L}^{p}(0,\tau; V)}^p \le \eta_\tau^p \| v \|_{{L}^{p}(0,\tau; V)}^p, \; \forall v \in {W}^{1,p}_{0,+}(0,\tau; V),
	     \end{eqnarray}
	     where $\eta_\tau^p := \kappa_0^p \gamma_\tau^p \tau^{p-1}$ and $\lim_{\tau \to 0} \eta_\tau = 0$ for any $p > 1$.
	     \vspace{.1cm}
	     
	     Finally, using the density of ${W}^{1,p}_{0,+}(0,\tau;V)$ in ${L}^p_+(0,\tau; V)$ and the fact that the positive cone ${L}^p_+(0,\tau; V)$ is generating, we obtain the estimate \eqref{Estimate-wellposed}. This concludes the proof. 
	     
	     \subsection{Proof of Theorem \ref{Main1}}
	     Using \eqref{representation}, the initial/boundary-value control problem \eqref{Main-system}
	     can be reformulated as the following system
	     \begin{align}\label{distributed}
	     	\begin{cases}
	     		\dot{z}(t) =(A_{-1}+B\Gamma) z(t)+B Ku(t),& t> 0,\\
	     		{z(0)=z_0}.&
	     	\end{cases}
	     \end{align}
	     To prove the well-posedness of \eqref{distributed}, we first establish that the operator $\check{\calA} : {D}(\check{\calA}) \to X$ defined by
	     \begin{align}\label{Def.checkA}
	     	\check{\calA} x := (A_{-1} + B \Gamma)x, \qquad x \in {D}(\check{\calA}) := \left\{x \in {D}(\tilde{A}) : (A_{-1} + B \Gamma)x \in X\right\},
	     \end{align}
	     generates a $C_0$-semigroup $\calT$ on $X$. To this ends, we shall use {Proposition~\ref{theorem}}. In fact, let $C:=\Gamma\vert_{{D}(A)}$. According to Definition \ref{Def.triple}, it follows {from} Theorem \ref{S5.P1} that $(A,B,C)$ is a positive ${L}^p$-well-posed triple on $V,X,V$. 
	     \vspace{.1cm}
	     
	     Let us now establish the regularity of $(A,B,C)$. By \cite[Cor. 3.5]{StaffansWeiss2002}, the transfer function satisfies $H(\la)=\Gamma D_\la$ for all ${\mathrm{R}\mathrm{e}( \lambda )}>\omega_0(T)$. Thus, using \eqref{regularity}, we obtain 
	     \begin{align}\label{Limit-transfer}
	     	\lim_{\la \to +\infty}\Vert H(\la)v\Vert_V=0, \qquad \forall v\in V.
	     \end{align} 
	     Hence, according to Definition \ref{Definition-regularity}, $(A,B,C)$ is a positive ${L}^p$-well-posed strongly regular triple with zero feedthrough. 
	     
	     Now, we define the operator 
	     \begin{eqnarray*}
	     	(\mathcal{F}^Av)(s):=\Gamma \Phi_s^A v, \quad \forall s\in [0,t],\ v\in  {W}^{1,p}_0(0,t;V).
	     \end{eqnarray*}
	     It follows from \eqref{Estimate-wellposed} that the operator $\mathcal{F}^A$ extends uniquely to a linear bounded positive operator (still denoted by $\mathcal{F}^A$) from ${L}^{p}(0,\t;V)$ to itself for each $\t> 0$. Since the triple  $(A,B,C)$ is strongly regular, it follow from Remark \ref{Yosida-extension} that this extension is given by
	     \begin{align*}
	     	(\mathcal{F}^A v)(t)=C_\Lambda \Phi_t^A v,  
	     \end{align*}
	     for all $v\in {L}^p_{loc}(0,\tau;V)$ and a.e. $t\ge 0$. Furthermore, by \eqref{Estimate-wellposed}, one has 
	     \begin{eqnarray*}
	     	\left\Vert \mathcal{F}^Av \right\Vert_{{L}^{p}(0,\t;V)}\leq \eta_\t \Vert v\Vert_{{L}^{p}(0,\t;V)},\quad \forall v\in {L}^{p}(0,\t;V),
	     \end{eqnarray*}
	     where $\lim_{\tau\to 0}\eta_\tau=0$ for any $p>1$. Therefore,  
	     \begin{align*}
	     	\lim_{\t\to 0}\Vert \mathcal{F}^A\Vert_{\calL({L}^{p}(0,\t;V))}=0,
	     \end{align*}
	     and thus there exists $\t_0>0$ small enough such that $\Vert \mathcal{F}^A\Vert_{\calL({L}^{p}(0,\t_0;V))} <1$. 
	     \vspace{.1cm}
	     
	     Applying Proposition \ref{theorem} together with \cite[Thm. A.1]{El3}, we obtain that the operator $\mathcal{A}$ defined in \eqref{def:A} generates a positive $C_0$-semigroup $T^{cl}$ on $X$ given by \eqref{variation}. By an argument similar to the proof of \cite[Thm. 4.1]{HMR}, one can show that $\check{\mathcal{A}}=\mathcal{A}$; hence $\check{\mathcal{A}}$ generates a positive $C_0$-semigroup $\mathcal{T}$ on $X$ given by $\mathcal{T}=T^{cl}$. Moreover, by {Proposition \ref{theorem}}, $(\mathcal{A},B,C_\Lambda)$ is a positive $L^p$-well-posed strongly regular triple. Therefore, system \eqref{distributed} (and hence \eqref{Main-system}) is equivalent to the following system
	     \begin{eqnarray*}
	     	\dot{z}(t) =\calA_{-1} z(t)+\calB u(t),\quad t> 0,\qquad
	     	{z(0)=z_0},
	     \end{eqnarray*}
	     where $\calB:=B K\in \calL(U,W_{-1})$. Since $K\in \calL(U,V)$, it follows from \cite[Prop. 2.5]{JaNPS} that for every ${z_0}\in X$ and every $u\in {L}^q_{loc}(\R_+,U)$ (with $q\ge p$), the initial/boundary-value control problem \eqref{Main-system} admits a unique solution $z(\cdot)\in { C}(\R_+;X)$ given by \eqref{variation1}. Moreover, for every $\tau>0$ there exists {$c_{\tau,q}>0$} such that
	     \begin{align}\label{CD}
	     	\Vert z(\t)\Vert \le c_{\tau,q}\left(\Vert {z_0}\Vert +\Vert u\Vert_{{L}^q(0,\t;U)} \right),
	     \end{align}
	     for all ${z_0}\in X$ and $u\in {L}^q_{loc}(\R_+,U)$. This completes the proof.

	     \begin{remark}
	     We note that the exponent $q$ is chosen such that $q \ge p$, where $p$ is the exponent of the ${L}^p$-well-posedness of the triple $(A,B,C)$. In particular, for $q=\infty$,  the estimate \eqref{CD} implies that the solution of the system \eqref{Main-system} depends continuously on the initial data $x$ and the input $u$.
	     \end{remark}

        \section{Application: ISS of coupled heat equations with delays and boundary disturbances}\label{Sec:5}
        As a concrete application, we analyze the exponential ISS of a system of three coupled heat equations. In this model, the boundary temperature of each component is fed, after a time delay, into the boundary heat flux of the next component, subject to boundary disturbances. More precisely, we consider the following system (see Figure \ref{fig1}) for $j\in \{1,2,3\}$:
         \begin{align}\label{Heat-equation}
        		\begin{cases}
        			\dot{z}_j(t,x)= a_j\partial_{xx}z_j(t,x)-b_j z_j(t,x),&t\geq 0, \;x\in (0,1),\cr  
        			z_j(0,x)= f_j(x), & x\in (0,1),\\
        			a_j\partial_x z_j(t,1)= c_j z_j(t-r_j,0)+ w_j(t), & t\geq 0,\\
        			\partial_x z_j(t,0)=0,&\\
        			z_j(\theta,0)=\varphi_j,& \theta\in [-r_j,0],
        		\end{cases}
        \end{align}
       where $a_j,b_j,c_j$ are positive constants. Here, $z_j(t,x)$ denotes the temperature distribution in the $j$-th component, $r_j\in (0,\infty)$ are transmission delays, $w_j$ are locally essentially bounded boundary disturbances, and $f_j,\varphi_j$ are given initial conditions.
        
        \begin{figure}[h]
        	\begin{center}
        		\begin{tikzpicture}[scale=0.5]
        			\node[element] (M) at (-15,0) {\bf Heat 1};
        			\node[element] (N) at (-9,0) {\bf Heat 2};
        			\node[element] (O) at (-3,0) {\bf Heat 3};
        			\node[fleche1] (P) at (-6,-2) {$\textcolor{red}{w_3(t)}$};
        			\node[fleche1] (T) at (-12,-2) {$\textcolor{red}{w_2(t)}$};
        			\draw[fleche] (M) -- (N) node[midway, above]{\small $z_1(0,t)$};
        			\draw[fleche] (N) -- (O) node[midway, above]{\small $z_2(0,t)$};
        			\draw[fleche1] (P) -- (-6,-0.1) node[midway, above]{};
        			\draw[fleche1] (T) -- (-12,-0.1) node[midway, above]{};
        			\draw[fleche1] (-18,0) -- (M) node[midway, above]{$\textcolor{red}{w_1(t)}$};
        		\end{tikzpicture}
        	\end{center}
        	\caption{Coupled heat equations with boundary disturbances.}\label{fig1}
        \end{figure}
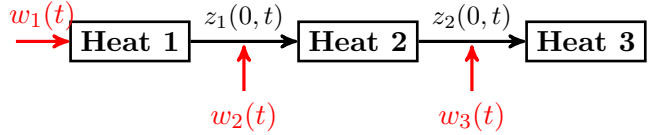
        
        In what follows, we investigate the well-posedness and stability of the system \eqref{Heat-equation} within the framework of ISS. To this end, we first reformulate the system as an initial/boundary-value control problem. In fact, consider the Banach space
        \begin{align*}
        	X:= \left({L}^1(0,1;\R)\right)^3,\qquad \Vert f\Vert_{X}:=\sum_{j=1}^{3}\Vert f_j\Vert_{{L}^1(0,1)}.
        \end{align*}
        Note that ${L}^1$ is the natural space for system \eqref{Heat-equation}, since the ${L}^1$-norm corresponds to the total heat flux.
        
        On $X$, we define the differential operator
        \begin{align*}
        	&\tilde{A}_h f:= \mathsf{a}\partial_{xx} f- \mathsf{b} f, \\
        	& f\in {D}(\tilde{A}_h):= \left\{f\in \left({W}^{2,1}(0,1)\right)^3: \partial_x f(0)=0\right\}\nonumber,
        \end{align*}
        where $\mathsf{a}:=\operatorname{diag}(a_j)_{j=1}^{j=3}$, $\mathsf{b}:=\operatorname{diag}(b_j)_{j=1}^{j=3}$.
        \vspace{.1cm}
        
        Next, let $\check{z}_j$ denote the boundary values of the heat distribution $z_j$ at the endpoints $0$, i.e., 
        \begin{eqnarray*}
        	\check{z}_j(t):=z_j(t,0), \qquad t\ge 0,\ j=1,2,3.
        \end{eqnarray*}
        For each $t\ge 0$, we consider the history segment
        \begin{eqnarray*}
        	\check{z}^t_j:[-r_{j},0] \ni \theta\to \R,\; \check{z}^t_j(\theta):=\check{z}_j(t+\theta), \quad j=1,2,3.
        \end{eqnarray*}
        
        We now introduce the Banach space 
        \begin{eqnarray*}
        	Y:=\prod_{j=1}^3{L}^1(-r_j,0;\R), \quad \Vert \varphi\Vert_{Y}:=\sum_{j=1}^{3}\Vert \varphi_j\Vert_{{L}^1(-r_j,0)}. 
        \end{eqnarray*}
        On $Y$, we define the differential operator
        \begin{eqnarray*}
        	\tilde{A}_\theta \varphi:=  \partial_\theta \varphi, \quad
        	\varphi\in  {D}(\tilde{A}_{\theta}):= \prod_{j=1}^3{W}^{1,1}(-r_j,0,\R).
        \end{eqnarray*}
        Recall that the restriction
        \begin{eqnarray*}
        	A:=	\tilde{A}_\theta, \quad {D}(A_{\theta}):=\left\{\varphi\in  {D}(\tilde{A}_{\theta}): \varphi(0)=0\right\}
        \end{eqnarray*}
        generates a positive $C_0$-semigroup $T_\theta:=(T_\theta(t))_{t\ge 0}$ on $Y$, given by 
        \begin{align*}
        	(T_\theta(t)\varphi)_j(\theta)=\begin{cases} 0,& -t\le \theta\le 0,\cr \varphi_j(t+\theta),& -r_j\le \theta\le -t,\end{cases}
        \end{align*}
        for all $j\in \{1,2,3\}$, $t\ge 0$, $\varphi\in Y$, and a.e. $\theta\in [-\max_j r_j,0]$.
        \vspace{.1cm}
        
        Following \cite{El3}, we introduce the new state variable
        \begin{align*}
        	\zeta(t)=\left((z_j(t,\cdot))_{j=1}^{j=3}, (\check{z}^{t}_j(\cdot))_{j=1}^{j=3}\right)^{\top},\quad t\ge 0.
        \end{align*}
        System \eqref{Heat-equation} can then be reformulated on $X \times Y$ as
        \begin{eqnarray}\label{Main-system1}
        	\begin{cases}
        		\dot{\zeta}(t)= \tilde{A}\zeta (t),& t> 0,\\
        		G \zeta(t)=\Gamma \zeta(t)+Ku(t), & t> 0,
        	\end{cases}
        \end{eqnarray}
        with initial condition $\zeta(0)=(f,\varphi)^{\top}$ and input function $u(t)=(w_j(t))^{j=3}_{j=1}$, where $f:=(f_j)_{j=1}^{j=3}$ and $\varphi:=(\varphi_j)_{j=1}^{j=3}$. Here, the operators $\tilde{A},G,\Gamma$ are defined by  
        \begin{eqnarray}\label{G,Ga}
        	\begin{array}{lll}
        		&\tilde{A}(f,\varphi)^{\top}:=(\tilde{A}_h f,\tilde{A}_\theta\varphi)^\top,\\
        		& G(f,\varphi)^{\top}:=\left((a_j\partial_{x}f_j(1))_{j=1}^{j=3},\varphi(0)\right)^{\top}, \\
        		&\Gamma(f,\varphi)^{\top}:=\left((\varphi_j(-r_j))_{j=1}^{j=3},(c_jf_j(0))_{j=1}^{j=3}\right)^{\top},
        	\end{array}
        \end{eqnarray}
        for all $(f,\varphi)^{\top}\in {D}(\tilde{A}):={D}(\tilde{A}_h)\times {D}(\tilde{A}_\theta)$, and 
        $$K:= \begin{pmatrix}
        	I\\0
        \end{pmatrix}.$$
        
        In view of \cite[Def. 2.7]{JaNPS} and \cite[Def. 3.17]{Andri}, we state the following definition of exponential ISS.
        \begin{definition}\label{ISS-def}
        	\begin{enumerate}
        		\item We say that \eqref{Heat-equation} is \emph{exponentially {input-to-sate stable (ISS)} with linear gain} if there exist constants $M,\omega,\kappa>0$ such that for every $t\ge 0$, $f\in X$, $\varphi\in Y$, and $u\in {L}^\infty(\R_+; \R^3)$,
        		\begin{enumerate}
        			\item $z\in { C}(\R_+;X)$, and
        			\item {\small $\Vert z(t)\Vert_{X} \le M e^{-\omega t}\left(\Vert f\Vert_{X} + \Vert \varphi\Vert_Y\right) + \kappa \Vert u\Vert_{{L}^\infty(\mathbb{R}_+; \mathbb{R}^3)}$}.
        		\end{enumerate}
        		
        		\item We say that \eqref{Main-system1} is \emph{exponentially ISS with linear gain} if there exist constants $M,\omega,\kappa > 0$ such that for every $t \ge 0$, $f \in X$, $\varphi \in Y$, and $u \in {L}^\infty(\R_+;\R^3)$,
        		\begin{enumerate}
        			\item $\zeta\in { C}(\R_+;X\times Y)$, and
        			\item {\small $\Vert \zeta(t)\Vert_{X \times Y} \le M e^{-\omega t} \Vert (f, \varphi)^{\top}\Vert_{X \times Y} + \kappa \Vert u\Vert_{{L}^\infty(\mathbb{R}_+; \mathbb{R}^3)}$}.
        		\end{enumerate}
        	\end{enumerate}
        \end{definition}
        
        \begin{remark}\label{Comparaison}
        	It is clear from the above definition that if \eqref{Main-system1} is exponentially ISS, then the system of heat equations \eqref{Heat-equation} is also exponentially ISS.
        \end{remark}
        \vspace{.1cm}
        
        The well-posedness of the system \eqref{Heat-equation} follows from the following result.
        \begin{proposition}\label{Main2}
        	Let $a_j,b_j,c_j>0$ for all $j\in \{1,2,3\}$. Then, for all $f\in X$, $\varphi\in Y$, and $u\in {L}^{\infty}_{loc}(\R_+;\R^3)$, the system of heat equations \eqref{Heat-equation} has a unique mild solution $z(\cdot)\in { C}(\R_+;X)$. 
        \end{proposition}
        \begin{proof}
        	We apply Theorem \ref{Main1}. To this end, we associate with system \eqref{Main-system1} the operator
        	\begin{align}\label{Big-A}
        		\calA(f,\varphi)^{\top}:=\tilde{A}(f,\varphi)^{\top}, \qquad
        		D(\calA):=\left\{(f,\varphi)^{\top}\in {D}(\tilde{A}): G(f,\varphi)^{\top}=\Gamma(f,\varphi)^{\top}\right\},
        	\end{align}
        	where the operators $\tilde{A},G,\Gamma$ are defined in \eqref{G,Ga}. 
        	\vspace{.1cm}
        	
        	Let $A:=\tilde{A}\vert_{\ker G}$. By Lemma \ref{Tech1}, the operator
        	\begin{align*}
        		A=\operatorname{diag}(A_h ,A_\theta),\quad {D}(A)={D}(A_h)\times {D}(A_\theta)
        	\end{align*}
        	generates a positive $C_0$-semigroup $T:=(T(t))_{t\ge 0}$ on $X\times Y$ given by $T(t)=\operatorname{diag}(T_h(t),T_\theta(t))$ for all $t\ge 0$. Moreover, we have $G:=\operatorname{diag}(G_h,\delta_0)$ with $G_h$ being the operator defined in Lemma \ref{Tech2} and $\delta_0:D(\tilde{A}_{\theta})\to \R^3$ is the Dirac mass at $0$ defined by $\delta_0 \varphi=\varphi(0)$. By Lemma \ref{Tech2}, $G$ is surjective. A direct computation using Lemma \ref{Tech2} shows that the Dirichlet operator associated with $(\tilde{A},G)$ is given by
        	\begin{align*}
        		D_\lambda=\operatorname{diag}(D^h_\lambda,\varepsilon_\la), \qquad \forall \lambda\in \rho(A),
        	\end{align*}
        	where $D^h_\lambda$ is given by \eqref{Dirichlet-heat} and $\varepsilon_\lambda:\R^3\to Y$ is defined by
        	\begin{eqnarray*}
        		(\varepsilon_\la v)(\theta):=e^{\la \theta}v, \quad v\in \R^3,\ \la \in \C,\ \theta\in [-\max_j r_j,0].
        	\end{eqnarray*}
        	Clearly, $D_\la$ is positive for all sufficiently large $\la\in \R$. In addition, for all for all $d,v\in \R^3_+$ and $\lambda>0$,
        	\begin{align*}
        		\Vert \la D_\lambda (d,v)^\top\Vert_{X\times Y}\le \frac{\lambda}{\min_j b_j+\lambda}\Vert d\Vert_{\R^3}+ \Vert v\Vert_{\R^3}.
        	\end{align*}
        	Thus,
        	\begin{eqnarray*}
        		\sup_{\la\ge \la_0}\Vert \la D_\la \Vert_{\calL(\R^6,X\times Y)}<\infty, 
        	\end{eqnarray*}
        	for some $\la_0>0$ large enough. On the other hand, the operator $\Gamma$ defined by \eqref{G,Ga} is clearly positive. Since the state-space $X \times Y$ is an ${L}^1$-space, it follows from \cite[Lem. 3.2]{El} that for any $\alpha > 0$,
        	\begin{align*}
        		\int_{0}^{\alpha}\Vert \Gamma T(t)(f,\varphi)^\top \Vert_{\R^6} \mathrm{d}t\leq \gamma_\alpha\Vert (f,\varphi)^\top\Vert_{X\times Y},
        	\end{align*}
        	for all $(f,\varphi)^\top\in {D}_+(A)$ and some $\gamma_\alpha>0$.
        	\vspace{.1cm}
        	
        	Now, for $d,v\in \R^3$, we have 
        	\begin{align*}
        		\Gamma D_\lambda (d,v)^\top=\left((e^{-\lambda r_j}v_j)_{j=1}^{j=3},\left(\frac{c_j d_j}{a_j \mu_j \sinh(\mu_j)}\right)_{j=1}^{j=3}\right)^\top,
        	\end{align*}
        	where we recall that $\mu_j:=\sqrt{\frac{b_j+\lambda}{a_j}}$. Thus, 
        	\begin{align*}
        		\Vert \Gamma D_\lambda (d,v)^\top\Vert_{\R^3\times \R^3}\le \tau(\lambda) \Vert (d,v)^\top\Vert,
        	\end{align*}
        	for all $d,v\in \R^3$ and $\lambda>0$, where 
        	\begin{align*}
        		\tau(\lambda):= \left(\max_j e^{-\lambda r_j}+\max_j\frac{c_j}{a_j \mu_j \sinh(\mu_j)}\right).
        	\end{align*}
        	Hence,
        	\begin{align*}
        		\sup_{\Vert (d,v)^\top\Vert=1}\Vert \Gamma D_\lambda (d,v)^\top\Vert\le \tau(\lambda)\downarrow 0 \text{ as } \lambda\uparrow +\infty.
        	\end{align*}
        	\vspace{.1cm}
        	
        	All assumptions of Theorem \ref{Main1} are satisfied, hence $\mathcal{A}$ generates a positive $C_0$-semigroup $\mathcal{T}:=(\mathcal{T}(t))_{t\ge 0}$ on $X\times Y$. Furthermore, for all $f\in X$, $\varphi\in Y$, and $u\in {L}^{\infty}_{loc}(\R_+;\R^3)$, the initial/boundary-value control problem \eqref{Main-system1} has a unique mild solution $\zeta(\cdot,f,\varphi,u)\in { C}(\R_+;X\times Y)$. Hence, the system \eqref{Heat-equation} has a unique mild solution  $z\in { C}(\R_+;X)$ given by 
        	\begin{align}\label{Explicit}
        		z(t,f,\varphi,u)=\left[\zeta(t,f,\varphi,u)\right]_1
        	\end{align}
        	for all $t \ge 0$, $f \in X$, $\varphi \in Y$, and $u \in L^\infty_{loc}(\R_+; \R^3)$, where $\left[\zeta(t,f,\varphi, u)\right]_1$ denotes the first component of $\zeta(t,f,\varphi,u)$. This completes the proof. 
        \end{proof}
        
        \begin{remark}
        	In view of Definition \ref{Definition-regularity}, it follows from the proof of Proposition \ref{Main2} that the triple $(A, B, C)$ induced from the initial/boundary-value control problem \eqref{Main-system1}--with $B := (\lambda I-A_{-1}) D_\lambda$ and $C := \Gamma\vert_{{D}(A)}$--is a positive  ${L}^1$-well-posed uniformly regular triple on $(\mathbb{R}^3 \times \mathbb{R}^3, X \times Y,\mathbb{R}^3\times\mathbb{R}^3)$. 
        \end{remark}
        
        Now we can characterize the exponential ISS of \eqref{Heat-equation}.
        \begin{theorem}
        	Let $a_j,b_j,c_j>0$ for all $j\in \{1,2,3\}$. If 
        	\begin{align}\label{Condition-ISS}
        		c_j<\sqrt{a_j b_j}\sinh{\left(\sqrt{\frac{b_j}{a_j}}\right)}, \qquad \forall j\in \{1,2,3\},
        	\end{align}
        	then the system of heat equations \eqref{Heat-equation} is exponentially ISS.
        \end{theorem}
        \begin{proof}
        	First, note that according to the proof of Theorem~\ref{Main1}, the initial/boundary-value control problem \eqref{Main-system1} is equivalent to the system
        	\begin{eqnarray*}
        		\dot{\zeta}(t) =\calA_{-1} z(t)+\calB u(t),\quad t> 0,\qquad
        		\zeta(0)=(f,\varphi)^\top,
        	\end{eqnarray*}
        	where $\mathcal{A}_{-1}$ is the extension of $\mathcal{A}$ with ${D}(\mathcal{A}_{-1}) = X\times Y$, and $\mathcal{B}:\mathbb{R}^3 \to W_{-1}$ is defined by 
        	\begin{align*}
        		\calB:= \begin{pmatrix}
        			-(A_h)_{-1}D_\lambda^h\\ 0
        		\end{pmatrix}.
        	\end{align*}
        	By \cite[Prop. 2.10]{JaNPS} (see also \cite[Thm. 3.18]{Andri}), system \eqref{Main-system1} is exponentially ISS if and only if $\omega_0(\calT)<0$ and $(\calA,\calB)$ is ${L^\infty}$-admissible. Proposition \ref{Main2} yields the $L^\infty$-admissibility, so it remains to show $\omega_0(\mathcal{T}) < 0$. In fact, since $X\times Y$ is ${L^1}$-space, it follows from \cite[Thm. C-IV.1.1]{Nagel} and \cite[Thm. A.2]{El3} that $\omega_0(\calT)<0$ if and only if $\omega_0(T)<0$ and $r(\Gamma D_0)<1$. On the one hand, $T=\operatorname{diag}(T_h,T_\theta)$, and {noting that} $\omega_0(T_\theta)=-\infty$ and Lemma \ref{Tech1}  gives $\omega_0(T_h)< 0$, we obtain $\omega_0(T)<0$. On the other hand, we have
        	\begin{align}\label{factorization}
        		\Gamma D_0&=\begin{pmatrix}
        			0 & I\\
        			M &  0
        		\end{pmatrix},
        	\end{align}
        	where $M:=\operatorname{diag}\left(\frac{c_j }{a_j\mu_j\sinh(\mu_j)}\right)_{j=1}^{j=3}$ with $\mu_j=\sqrt{\frac{b_j}{a_j}}$. Thus, 
        	\begin{align*}
        		r(\Gamma D_0)=\max_j \sqrt{\frac{c_j }{a_j\mu_j\sinh(\mu_j)}}.
        	\end{align*}
        Hence, under the condition \eqref{Condition-ISS}, one has $r(\Gamma D_0)< 1$, and therefore $\omega_0(\calT) < 0$. Thus, according to \cite[Prop. 2.10]{JaNPS}, system \eqref{Main-system1} is exponentially ISS. Finally, Remark \ref{Comparaison} implies that the system of heat equations \eqref{Heat-equation} is exponentially ISS, which completes the proof. 
        \end{proof}

        \section{Conclusion}\label{Sec:6}
        In this paper, we investigated the well‑posedness of a class of boundary control systems whose boundary conditions involve an unbounded linear operator that is not necessarily closed or closable--a structure that arises naturally in PDE models with boundary feedback or coupling. For such systems, we established direct and explicit conditions on the system data guaranteeing continuous dependence of solutions on initial data and $L^p$-inputs for any $p \in [1,\infty]$. A key contribution is a novel boundedness estimate for input/output maps of LTI infinite-dimensional systems with unbounded input and output operators. These results were then applied to investigate the exponential ISS for a system of boundary coupled heat equations on $L^1$-space. In particular, we obtained specific conditions on the coefficients of the heat equations that yield exponential ISS. In future work, we plan to extend this analysis to nonlinear evolution equations.

        \appendix
        \section{}\label{App}
        Here we present technical results used in the proof of Proposition \ref{Main2}.
        
        \begin{lemma}\label{Tech1}
        	Let $a_j,b_j>0$ for all $j\in \{1,2,3\}$. Then, the operator 
        	\begin{align*}
        		 A_h f := (a_j \partial_{xx} f_j-b_j f_j)_{j=1}^{j=3}, \qquad
        		 {D}(A_h) = \{f \in (W^{2,1}(0,1))^3 : \partial_x f(0) = \partial_x f(1) = 0\} 
        	\end{align*}
        	generates a positive $C_0$-semigroup $T_h:= (T_h(t))_{t\ge 0}$ on $X$ such that $\omega_0(T_h) < 0$.
        \end{lemma}
        \begin{proof}
        	First, we show that $A_h$ generates a positive $C_0$-semigroup $T_h$ on $X$. We split $A_h$ as
        	\begin{eqnarray*}
        		A_h f = (a_j\partial_{xx}f_j)_{j=1}^{j=3} - (b_jf_j)_{j=1}^{j=3}=: A_h^1 + A_h^2.
        	\end{eqnarray*}
        	We shall use \cite[Thm. C-II.1.2]{Nagel}. First, we prove that $A_h^1$ is {dispersive}. In fact, we have 
        	\begin{eqnarray*}
        		{D}(A_h^1)=\left\{ f\in \left({W}^{2,1}(0,1)\right)^3: \partial_x f(0)=\partial_x f(1)=0 \right\}.
        	\end{eqnarray*}
        	For $f \in {D}(A_h^1)$, define for each $j$:
        	\begin{align*}
        		\phi_j(x)= \begin{cases}
        			1,& {\rm if }\; f_j(x)>0,\\
        			0,& {\rm if \, not}.
        		\end{cases}
        	\end{align*}
        	Then $\phi = (\phi_j)_{j=1}^{j=3} \in  \left(L^\infty(0,1)\right)^3=X'$, with $\phi_j \geq 0$,  $\|\phi_j\|_\infty = 1$, and $\langle f,\phi \rangle_{X\times X'}=\Vert f_+\Vert_{X}$. 
        	
        	Now, fix $j$ and let $\Omega_j = \{x\in (0,1): f_j(x)> 0\}$. Since $f_j$ is continuous, $\Omega_j$ is open and can be written as a disjoint union of open intervals $\Omega_j = \bigcup_{k=1}^{m} (\alpha_k, \beta_k) \subset (0,1)$. On each interval $(\alpha_k, \beta_k)$, we have $f_j(x)> 0$, $f_j(\alpha_k)= f_j(\beta_k)=0$, $\partial_x f_j(\alpha_k) \ge 0$, and $ \partial_xf_j(\beta_k) \le 0$. Hence,
        	\begin{align*}
        		\int_{\alpha_k}^{\beta_k} \partial_{xx} f_j(x) {\mathrm{d}x} = \partial_xf_j(\beta_k) -\partial_x f_j(\alpha_k)\le 0,
        	\end{align*}
        	for all $k\in \N$ and $j=1,2,3$. Therefore,
        	\begin{align*}
        		\langle A_h^1 f,\phi \rangle_{X\times X'}&=\sum_{j=1}^{3}a_j\sum_{k\in \N}\int_{\alpha_k}^{\beta_k}  \partial_{xx} f_j(x){\mathrm{d}x}\\
        		& = \sum_{j=1}^3 a_j \sum_{k\in \N} \partial_xf_j(\beta_k) -\partial_x f_j(\alpha_k)\le 0,
        	\end{align*}
        	which shows that $A_h^1$ is {dispersive}. 
        	
        	Next, we show that $(I-A_h^1)$ is surjective. For $g\in X$, we need to find $f\in D(A_h^1)$ such that $(I - A_h^1)f = g$. This reduces to solving for each $j$:
        	\begin{eqnarray*}
        		f_j - a_j \partial_{xx} f_j = g_j, \quad \partial_x f_j(0) =\partial_x f_j(1) = 0.
        	\end{eqnarray*}
        	Let $\gamma_j=\sqrt{\frac{1}{a_j}}$. Then, the solution is given by
        	 	\begin{align*}
        			f_j(x) = \frac{1}{2a_j \gamma_j} \left[ e^{\gamma_j x} \int_x^{1} e^{-\gamma_j s} g_j(s) {\mathrm{d}s} + e^{-\gamma_j x} \int_0^{x} e^{\gamma_j s} g_j(s){\mathrm{d}s} \right] + c_j,
        	\end{align*}
        where $c_j$ is chosen to satisfy the boundary conditions. One can verify that $f_j \in {W}^{2,1}(0,1)$ and satisfies the Neumann conditions. Hence $(I-A_h^1)$ is surjective. By \cite[Thm. C-II.1.2]{Nagel}, $A_h^1$ generates a contraction semigroup $T_h^1$ on $X$. The operator $A_h^2 f :=-(b_j f_j)_{j=1}^{j=3}$ with domain ${D}(A_h^2)=X$ generates the positive multiplication semigroup $T_h^2(t)f = (e^{-b_j t} f_j)_{j=1}^{j=3}$, which is a contraction semigroup since $b_j> 0$. By applying the Trotter product formula (cf. \cite[Corollary III.5.8]{engel2000}), we obtain that $A_h$ generates a $C_0$-semigroup $T_h$ given by 
        	\begin{align*}
        		T_h(t)f=\lim_{n\to +\infty}\left[T_h^1\left(\tfrac{t}{n}\right)T_h^2\left(\tfrac{t}{n}\right)\right]^n f
        	\end{align*} 
        	for all $f\in X$. Since $X_+$ is closed, it follows from the above formula that $T_h$ is also positive. 
        	
        	Finally, we show that $\omega_0(T_h)<0$. According to the proof of \cite[Lem. 5.2]{El2}, for any \(g \in X\) the equation $-A_h f=g$ admits a unique solution $f\in {D}(A_h)$ given componentwise by
        	 \begin{align*}
        			f_j(x)&=\frac{\int_{0}^{1}\sinh\left(\sqrt{\tfrac{b_j}{a_j}}(1-y)\right)a_j^{-1}g_j(y)dy}{-\sqrt{\tfrac{b_j}{a_j}} \sinh\left(\sqrt{\tfrac{b_j}{a_j}}\right)}\cosh{\left(\sqrt{\tfrac{b_j}{a_j}}\right)}\\
        			&+\left(\sqrt{\tfrac{b_j}{a_j}}\right)^{-1}\int_{0}^{x}\sinh\left(\sqrt{\tfrac{b_j}{a_j}}(x-y)\right)a_j^{-1}g_j(y)dy.
        	\end{align*}
        for $x\in (0,1)$ and $j=1,2,3$. Since $T_h$ is positive, then $s(A_h)<0$. Applying the spectral mapping theorem for positive semigroups on ${L}^1$-spaces \cite[Thm. C-IV.1.1]{Nagel}, we conclude that $\omega_0(T_h)< 0$, which completes the proof.
        \end{proof}
        
        \begin{lemma}\label{Tech2}
        	Let $a_j,b_j>0$ for all $j\in \{1,2,3\}$. Let $G_h$ be the operator defined by 
        	\begin{align*}
        		G_h f:= (a_j\partial_{x}f_j(1))_{j=1}^{j=3}, \quad f\in {D}(\tilde{A}_h).
        	\end{align*}
        	Then, the Dirichlet operator $D_\la^h$ associated with $(\tilde{A}_h,G_h)$ is given by 
        	\begin{align}\label{Dirichlet-heat}
        		(D_\lambda^hd)(x)=\dfrac{d_j}{a_j \mu_j \sinh(\mu_j)} \cosh(\mu_j x), 
        	\end{align}
        	for all ${\mathrm{R}\mathrm{e}( \lambda)}>-\bar{b}$, $x\in (0,1)$, and $j\in \{1,2,3\}$, where $\bar{b}=\max_j b_j$ and $\mu_j:=\sqrt{\frac{b_j+\lambda}{a_j}}$.
        	In addition, $D_\la$ is positive for all sufficiently large $\la\in \R$.
        \end{lemma}
        
        \begin{proof}
        	For $d \in \R^3$, we solve the abstract elliptic problem
        	\begin{align}\label{ADP}
        		(\lambda-\tilde{A}_h)f=0,\qquad G_h f=d.
        	\end{align}
        	Let $\lambda>-\bar{b}$, so that $\lambda+b_j> 0$ for all $j\in \{1,2,3\}$. By direct computation we obtain the problem \eqref{ADP} has a unique solution given by 
        	\begin{align*}
        		f_j(x) = \dfrac{d_j}{a_j \mu_j \sinh(\mu_j)} \cosh(\mu_j x),
        	\end{align*}
        	for $x\in (0,1)$ and $j=1,2,3$. This shows that for any $d\in \mathbb{R}^3$, there exists a function $f\in D(\tilde{A}_h)$ satisfying \eqref{ADP}. Thus $G_h$ is surjective. By definition, the Dirichlet operator is the solution of \eqref{ADP}, and thus \eqref{Dirichlet-heat} holds true.
        	
        	Finally, notice that when $\lambda$ is sufficiently large, $\mu_j> 0$, $\sinh(\mu_j)>0$, and $\cosh(\mu_jx) > 0$ for all $x \in [0,1]$. Thus, if $d_j\geq 0$, then $f_j(x) \geq 0$ for all $x$, so $D_\lambda^h$ maps $\mathbb{R}_+^3$ into $X_+$. This completes the proof.
        \end{proof}


	\bibliographystyle{abbrvnat}
	\bibliography{cas-refs}

@book{CHARALAMBOS,
	author = {Aliprantis, Charalambos D and O. Burkinshaw},
	title = {Positive operators},
	publisher = {Springer},
	address = {Dordrecht, Netherlands},
	year = {2006}
}

@article{SGPS,
author = {Arora, S. and Gl{\"u}ck, J. and Paunonen, L. and Schwenninger, F-L.},
title = {Limit-case admissibility for positive infinite-dimensional systems},
journal = {J. Differ. Equations},
volume = {440},
pages = {113435},
year = {2025},
issn = {0022-0396},
doi = {https://doi.org/10.1016/j.jde.2025.113435}
}

@book{BFR,
	author = {A. B\'{a}tkai and Fijav\v{z}, Marjeta Kramar and A. Rhandi},
	title = {Positive operator semigroups: from finite to infinite dimensions},
	publisher = {Birkh\"{a}user-Verlag},
	address = {Basel},
	year = {2016}
}

@article{BJVW,
	author = {A. B\'{a}tkai and B. Jacob and J. Voigt and J. Wintermayr},
	title = {Perturbation of positive semigroups on {AM}-spaces},
	journal = {Semigroup Forum},
	volume = {96},
	year = {2018},
	pages = {33-347}
}

@article{ChM,
	author = {A. Chen and K. Morris},
	title = {Well-posedness of boundary control systems},
	journal = {SIAM J. Control Optim.},
	volume = {42},
	year = {2003},
	pages = {1244-1265}
}

@incollection{Desch,
	author = {W. Desch and W. Schappacher},
	title = {Some generation results for perturbed semigroups},
	booktitle = {Semigroup Theory and Applications (Proceedings, Trieste, 1987)},
	editor = {Cl\'{e}ment, P. and Invernizzi, S. and Mitidieri, E. and Vrabie, I-I. },
	series = {Lecture Notes in Pure and Appl. Math.},
	volume = {116},
	publisher = {Dekker},
	address = {New York},
	year = {1989},
	pages = {125-152}
}

@article{El2,
	author = {El Gantouh, Y.},
	title = {Boundary approximate controllability under positivity constraints of infinite-dimensional control systems},
	journal = {J. Optim. Theory Appl.},
	volume = {198},
	number = {2},
	year = {2023},
	pages = {449-478}
}

@article{El,
	author = {El Gantouh, Y.},
	title = {Positivity of infinite-dimensional linear systems},
	year = {2023},
	note = {arxiv.org/abs/2208.10617}
}

@article{El3,
  author = {{El Gantouh}, Y. and Liu, Y.},
  title = {Well-posedness and stability of boundary delay equations},
  journal = {Syst. Control Lett.},
  volume = {208},
  pages = {106307},
  year = {2026},
  doi = {10.1016/j.sysconle.2025.106307}
}

@article{elgantouh2024admissibility,
  author    = {El Gantouh, Y. and Liu, Y. and Lu, J. and Cao, J.},
  title     = {Admissibility of control operators for positive semigroups and robustness of input-to-state stability},
  year      = {2025},
  volume    = {arXiv:2503.04097},
  archivePrefix = {arXiv},
  eprint    = {2503.04097}
}

@article{Gr,
	author = {G. Greiner},
	title = {Perturbing the boundary conditions of a generator},
	journal = {Houston J. Math.},
	volume = {13},
	year = {1987},
	pages = {213-229}
}

@article{HMR,
	author = {S. Hadd and R. Manzo and A. Rhandi},
	title = {Unbounded perturbations of the generator domain},
	journal = {Discrete Contin. Dyn. Syst.},
	volume = {35},
	year = {2015},
	pages = {703-723}
}

@article{JaNPS,
	author = {B. Jacob and R. Nabiullin and Partington, Jonathan R and Schwenninger, Felix L},
	title = {Infinite-dimensional input-to-state stability and Orlicz spaces},
	journal = {SIAM J. Control Optim.},
	volume = {56},
	number = {2},
	year = {2018},
	pages = {868-889}
}

@article{JaSZ,
	author = {B. Jacob and Schwenninger, Felix L and H. Zwart},
	title = {On continuity of solutions for parabolic control systems and input-to-state stability},
	journal = {J. Differ. Equations},
	volume = {266},
	number = {10},
	year = {2019},
	pages = {6284-6306}
}

@book{Lasiecka2000,
  author = {Lasiecka, I. and Triggiani, R.},
  title = {Abstract parabolic systems},
  booktitle = {Control Theory for Partial Differential Equations: Continuous and Approximation Theories. {I}},
  series = {Encyclopedia of Mathematics and its Applications},
  volume = {74},
  publisher = {Cambridge University Press},
  address = {Cambridge},
  year = {2000},
  isbn = {978-0-521-43408-3}
}

@book{Nagel,
	editor = {R. Nagel},
	title = {One-Parameter semigroups of positive operators},
	series = {Lecture Notes in Mathematics},
	publisher = {Springer},
	address ={Berlin},
	year = {1986}
}

@incollection{NaS,
	author = {R. Nagel and E. Sinestrari},
	title = {Inhomogeneous volterra integrodifferential equations for {H}ille-{Y}osida operators},
	booktitle = {Functional Analysis, Proceedings Essen, 1991},
	editor = {K.D. Bierstedt and A. Pietsch and W.M. Ruess and D. Vogt},
	series = {Lect. Notes in Pure and Appl. Math.},
	volume = {150},
	publisher = {Dekker},
	address = {New York},
	year = {1994},
	pages = {51-70}
}

@article{Salam,
	author = {D. Salamon},
	title = {Infinite-dimensional linear system with unbounded control and observation: a functional analytic approach},
	journal = {Trans. Am. Math. Soc.},
	volume = {300},
	year = {1987},
	pages = {383-431}
}

@book{Schaf,
	author = {Schaefer, Helmut H},
	title = {Banach lattices and positive operators},
	publisher = {Springer-Verlag},
	address = {Berlin-Heidelberg},
	year = {1974}
}

@book{Sta,
	author = {Staffans, Olof J},
	title = {Well-posed linear systems},
	publisher = {Cambridge University Press},
	address = {Cambridge},
	year = {2005}
}

@book{TW,
	author = {M. Tucsnak and G. Weiss},
	title = {Observation and control for operator semigroups},
	publisher = {Birkh\"{a}user},
	address = {Basel, Boston, Berlin},
	year = {2009}
}

@article{WF,
	author = {G. Weiss},
	title = {Regular linear systems with feedback},
	journal = {Math. Control Signals Syst.},
	volume = {7},
	year = {1994},
	pages = {23-57}
}

@article{CZZ,
	author = {Chen, Q. and Zheng, J. and Zhu, G.},
	title = {Backstepping control of a class of space-time-varying linear parabolic {PDEs} via time invariant kernel functions},
	journal = {SIAM J. Control Optim.},
	volume = {62},
	number = {6},
	year = {2024},
	pages = {2992-3018}
}

@book{engel2000,
	author    = {Klaus-Jochen Engel and Rainer Nagel},
	title     = {One-parameter semigroups for linear evolution equations},
	series    = {Graduate Texts in Mathematics},
	volume    = {194},
	publisher = {Springer-Verlag},
	address   = {New York},
	year      = {2000},
	isbn      = {978-0-387-98463-0}
}

@article{StaffansWeiss2002,
	author = {Staffans, O. and Weiss, G.},
	title = {Transfer functions of regular linear systems. {Part II}: The system operator and the {Lax-Phillips} semigroup},
	journal = {Trans. Amer. Math. Soc.},
	volume = {354},
	pages = {3229--3262},
	year = {2002}
}

@article{Andri,
  author = {Mironchenko, A. and Prieur, C.},
  title = {Input-to-state stability of infinite-dimensional systems: recent results and open questions},
  journal = {SIAM Review},
  volume = {62},
  number = {3},
  pages = {529--614},
  year = {2020},
}

@article{Fattorini1968,
  author = {Fattorini, H. O.},
  title = {Boundary control systems},
  journal = {SIAM Journal on Control},
  volume = {6},
  number = {3},
  pages = {349--385},
  year = {1968},
}

@article{WR,
  author  = {Weiss, G.},
  title   = {Transfer functions of regular linear systems. {Part I}: Characterization of regularity},
  journal = {Trans. Amer. Math. Soc.},
  volume  = {342},
  pages   = {827--854},
  year    = {1994}
}

@book{Jacob2012,
  author = {Jacob, B. and Zwart, H.},
  title = {Linear {P}ort-{H}amiltonian systems on infinite-dimensional spaces},
  series = {Operator Theory: Advances and Applications},
  volume = {223},
  pages = {143},
  year = {2012},
  publisher = {Springer},
  address = {Basel}
}

@book{Curtain1995,
  author = {Curtain, R. F. and Zwart, H.},
  title = {An introduction to infinite-dimensional linear systems theory},
  series = {Texts in Applied Mathematics},
  volume = {21},
  publisher = {Springer-Verlag},
  address = {New York},
  year = {1995}
}

@article{Emirsajlow2000,
  author = {Emirsajlow, Z. and Townley, S.},
  title = {From {PDEs} with boundary control to the abstract state equation with an unbounded input operator: a tutorial},
  journal = {Eur. J. Control},
  volume = {6},
  number = {1},
  pages = {27--53},
  year = {2000}
}

@article{Tucsnak2014,
  author = {Tucsnak, M. and Weiss, G.},
  title = {Well-posed systems--the {LTI} case and beyond},
  journal = {Automatica},
  volume = {50},
  pages = {1757--1779},
  year = {2014}
}

@article{Alessio2026,
	author = {Barbieri, A. and Engel, K. J},
title = {On struSomectured finite–rank perturbations of positive operator semigroups},
journal = {Evol. Equ. Control Theory},
volume = {16},
pages = {114--138},
year = {2026},
doi = {10.3934/eect.2025070},
}

@book{Karafyllis2019a,
  author = {Karafyllis, I. and Krstic, M.},
  title = {Input-to-state stability for {PDEs}},
  series = {Communications and Control Engineering},
  publisher = {Springer},
  year = {2019},
  isbn = {978-3-319-91010-4},
  doi = {10.1007/978-3-319-91011-6},
}

@article{Lhachemi2019,
  author = {Lhachemi, H. and Shorten, R.},
  title = {{ISS} property with respect to boundary disturbances for a class of Riesz-spectral boundary control systems},
  journal = {Automatica},
  volume = {109},
  pages = {108504},
  year = {2019},
  doi = {10.1016/j.automatica.2019.108504},
}

@article{SZZ,
  author = {Sun, X. and Zheng, J. and Zhu, G.},
  title = {Finite-time input-to-state stability for infinite-dimensional systems},
  journal = {Int. J. Robust Nonlinear Control},
  volume = {35},
  pages = {3141--3153},
  year = {2025},
  doi = {10.1002/rnc.7829},
}
	
\end{document}